\documentclass{IOS-Book-Article}

\usepackage{makeidx}
\usepackage{graphicx}
\usepackage{latexsym}
\usepackage{xspace}
\usepackage{euscript}
\usepackage{amsmath}
\usepackage{color}
\usepackage{amssymb}
\usepackage[latin1]{inputenc}

\usepackage[latin1]{inputenc}
\usepackage{times}
\normalfont
\usepackage[T1]{fontenc}

\usepackage{amsfonts}

\usepackage[latin1]{inputenc}
\usepackage{url,color}

\newtheorem{theorem}{Theorem}
\newtheorem{corollary}{Corollary}
\newtheorem{lemma}{Lemma}
\newtheorem{proposition}{Proposition}

\newtheorem{defi}{Definition}

\newenvironment{proof}{ \noindent \textit{Proof:} }
{\hfill$\Box$\endtrivlist}

\begin{document}

\begin{frontmatter}

\title{On the existence of Free Models in Fuzzy Universal Horn Classes}

\runningtitle{On the existence of Free Models in Fuzzy Universal Horn Classes}

\runningauthor{Costa and Dellunde}

\author[A,B]{\fnms{Vicent} \snm{Costa}},
\author[A,B,C]{\fnms{Pilar} \snm{Dellunde}}

\address[A]{Universitat Aut\`onoma de Barcelona\\
}

\vspace{-0.4cm}

\address[B]{Artificial Intelligence Research Institute (IIIA - CSIC)\\
Campus UAB, 08193 Bellaterra, Catalonia (Spain) \\
}

\vspace{-0.4cm}

\address[C]{Barcelona Graduate School of Mathematics\\
\email{vicent@iiia.csic.es}\\
\email{pilar.dellunde@uab.cat}\\}

\begin{abstract}
This paper is a contribution to the study of the universal Horn fragment of predicate fuzzy logics, focusing on some relevant notions in logic programming. We introduce the notion of \emph{term structure associated to a set of formulas} in the fuzzy context and we show the existence of free models in fuzzy universal Horn classes.  We prove that every equality-free consistent universal Horn fuzzy theory has a Herbrand model. 
\end{abstract}

\begin{keyword}  Horn clause \sep Free model  \sep Herbrand structure \sep  Predicate Fuzzy Logics.

\end{keyword}

\end{frontmatter}

\maketitle

\section{Introduction}
\label{Introduction} 

Since their introduction in \cite{Kin43}, Horn clauses have shown to have good logic properties and have proven to be of importance for many disciplines, ranging from logic programming, abstract specification of data structures and relational data bases, to abstract algebra and model theory. However, the analysis of Horn clauses has been mainly restricted to the sphere of classical logic. For a good exposition of the most relevant results concerning Horn clauses in classical logic we refer to \cite{Hod93}, and to \cite{Mak87} for a good study of their importance in computer science. 

The interest in continuous t-norm based logics since its systematization by H\'ajek  \cite{Ha98} and the subsequent study of core fuzzy logics \cite{CiHa10} invite to a systematic development of a model theory of these logics (and of algebraizable non-classical logics in general). Cintula and H\'ajek raised the open question of characterizing theories of Horn clauses in predicate fuzzy logics \cite{CiHa10}. Our first motivation to study the Horn fragment of predicate fuzzy logics was to solve this open problem, the present article is a first contribution towards its solution.

Some authors have contributed to the study of Horn clauses over fuzzy logic. In \cite{Be02,Be03,BeVy05,BeVic06,BeVic06b,Vy15} B\v{e}lohl\'avek and Vychodil study fuzzy equalities, they work with theories that consist of formulas that are implications between identities with premises weighted by truth
degrees. They adopt Pavelka style: theories are fuzzy sets of formulas and they consider degrees
of provability of formulas from theories. Their basic structure of truth degrees is a complete
residuated lattice. The authors derive a Pavelka-style completeness theorem (degree of provability
equals degree of truth) from which they get some particular cases by imposing restrictions
on the formulas under consideration. As a particular case, they obtain completeness of fuzzy
equational logic. In different articles they study the main logical properties of varieties of algebras with fuzzy equalities. Taking a different approach, in a series of papers \cite{Ge01b, Ge01,Ge05}, Gerla proposes to base fuzzy control on fuzzy logic programming, and observes that the class of fuzzy Herbrand interpretations gives a semantics for fuzzy programs. Gerla works with a complete, completely distributive, lattice of truth-values. For a reference on fuzzy logic programming see \cite{Voj01, Ebra01}.

Several definitions of Horn clause have been proposed in the literature of fuzzy logics, but there is not a canonical one yet. Cintula and H\'ajek affirm that the elegant approach of \cite{BeVic06} is not the only possible one. In \cite{DuPra96}, Dubois and Prade discuss different possibilities of defining \emph{fuzzy rules} and they show how these different semantics can be captured in the framework of fuzzy set theory and possibility theory. Following all these works, our contribution is a first step towards a systematic model-theoretic account of Horn clauses in the framework introduced by H\'ajek in \cite{Ha98}. We introduce a basic definition of Horn clause over the predicate fuzzy logic MTL$\forall^m$ that extends the classical one in a natural way. In future work we will explore different generalizations of our definitions for expanded languages. Our approach differs from the one of B\v{e}lohl\'avek and Vychodil because we do not restrict to fuzzy equalities. Another difference is that, unlike these authors and Gerla, our structures are not necessarily over the same complete algebra, because we work in the general semantics of \cite{Ha98}.

In the present work we have focused on the study of \emph{free models of Horn clauses}. Free structures have a relevant role in classical model theory and logic programming. Admitting free structures make reasonable the concepts of \emph{closed-word assumption} for databases and \emph{negation as failure} for logic programming. These structures allow also a procedural interpretation for logic programs (for a reference see \cite{Mak87}). Free structures of a given class are minimal from an algebraic point of view, in the sense that there is a unique homomorphism from these structures to any other structure in the class. The free structures introduced here are \emph{term structures}, structures whose domains consist of terms or equivalence classes of terms of the language. In classical logic, term structures have been used to prove the satisfiability of a set of consistent sentences, see for instance \cite[Ch.5]{EbiFlu94}. Notorious examples of term structures are Herbrand models, they play an important function in the foundations of logic programming. Several authors have been studied Herbrand models in the fuzzy context (for a reference see \cite{Ge05,Voj01,Ebra01}), providing theoretical background for different classes of fuzzy expert systems. For a general reference on Herbrand Theorems for substructural logics we refer to \cite{CiMet13}.

\smallskip

The present paper is an extension of the work presented in the 18th International Conference of the Catalan Association for Artificial Intelligence (CCIA 2015) \cite{CoDe15}. Our main original contributions are the following:

\begin{itemize}
  \item Introduction of the notion of term structure associated to a theory over predicate fuzzy logics. If the theory consist of universal Horn formulas, we show that the associated term structure is a model of the theory (Theorem 2).
  \item Existence of free models in fuzzy universal Horn classes of structures. In the case that the language has an equality symbol $\approx$ interpreted as a similarity, we prove the existence of models which are free in the class of reduced models of the theory (Theorem 1). In the case that the language has the crisp identity, the class has free models in the usual sense.
 \item Consistent universal Horn theories over predicate fuzzy logics (that contains only the truth-constants $\overline{1}$ and $\overline{0}$) have classical models (Corollary \ref{classic}).
  \item Introduction of Herbrand structures. We prove that every equality-free consistent universal Horn theory over predicate fuzzy logics have a Herbrand model (Corollary \ref{corollary H-model}).
\end{itemize}

The paper is organized as follows. Section 2 contains the preliminaries on predicate fuzzy logics. In Section 3 we introduce the definition of Horn clause over predicate fuzzy logics. In Section 4 we study the term structures associated to universal Horn theories. In Section 5 we introduce Herbrand structures for equality-free theories. Finally, there is a section devoted to conclusions and future work.

\section{Preliminaires}
\label{Preliminaires}

Our study of the model theory of Horn clauses is focused on the basic predicate fuzzy logic MTL$\forall^m$ and some of its extensions based on propositional core fuzzy logics in the sense of \cite{CiHa10}. The logic MTL$\forall^m$ is the predicate extension of the left-continuous t-norm based logic MTL introduced in \cite{EsGo01}, where MTL-algebras are defined as bounded integral commutative residuated lattices $(A,\sqcap,\sqcup,*,\Rightarrow,0,1)$, where $\sqcap$ and $\sqcup$ are respectively the lattice meet and join operations and $(\Rightarrow,*)$ is a residuated pair, satisfying the pre-linearity equation $(x\Rightarrow y)\sqcup(y\Rightarrow x)=1$ (for an exhaustive exposition of MTL-algebras, see \cite{NoEsGis05}). In addition, completeness of this logic with respect to MTL-algebras is proven in \cite[Th.1]{EsGo01}, and Jenei and Montagna shown that MTL is the logic of all left continuous t-norms and their residua \cite{JeMon02}. Now we present the syntax and semantics of predicate fuzzy logics and we refer to \cite[Ch.1]{CiHaNo11} for a complete and extensive presentation.

\begin{defi} [Syntax of Predicate Languages]
 A \emph{predicate language} $\mathcal{P}$ is a  triple $\left\langle Pred_{\mathcal{P}},Func_{\mathcal{P}},Ar_{\mathcal{P}} \right\rangle$, where $Pred_{\mathcal{P}}$ is a nonempty set of \emph{predicate symbols}, $Func_{\mathcal{P}}$ is a set of \emph{function symbols} (disjoint from $Pred_{\mathcal{P}}$), and $Ar_{\mathcal{P}}$ represents the \emph{arity function}, which assigns a natural number to each predicate symbol or function symbol. We call this natural number the \emph{arity of the symbol}. The predicate symbols with arity zero are called \emph{truth constants}, while the function symbols whose arity is zero are named \emph{individual constants} (\emph{constants} for short) or \emph{objects}. 
\end{defi}

The set of $\mathcal{P}$-terms, $\mathcal{P}$-formulas and the notions of free occurrence of a variable, open formula, substitutability and sentence are defined as in classical predicate logic. From now on, when it is clear from the context, we will refer to $\mathcal{P}$-terms and $\mathcal{P}$-formulas simply as \emph{terms} and \emph{formulas}. A term $t$ is \emph{ground} if it has no variables. Throughout the paper we consider the equality symbol as a binary predicate symbol, not as a logical symbol, that is, the equality symbol is not necessarily present in all the languages and its interpretation is not fixed. From now on, let $L$ be a core fuzzy logic in a propositional language $\mathcal{L}$ that contains only the truth-constants $\overline{1}$ and $\overline{0}$ (for an extended study of core fuzzy logics, see \cite{CiHa10}).

\begin{defi} We introduce an axiomatic system for the predicate logic $L\forall^m$:  
\begin{description}
 
 \item[($\mathrm{P}$)]$\space\space$ $\space\space$ $\space\space$ Instances of the axioms of $L$ (the propositional variables are substituted for first-order formulas). 
 
 \item[($\forall 1$)]$\space\space$ $(\forall x)\varphi(x)\rightarrow\varphi(t)$, where the term $t$ is substitutable for $x$ in $\varphi$.
 
 \item[($\exists1$)]$\space\space$ $\varphi(t)\rightarrow(\exists x)\varphi(x)$, where the term $t$ is substitutable for $x$ in $\varphi$.
 
 \item[($\forall 2$)]$\space\space$ $(\forall x)(\xi\rightarrow\varphi)\rightarrow(\xi\rightarrow(\forall x)\varphi(x))$, where $x$ is not free in $\xi$.
 
 \item[($\exists2$)]$\space\space$ $(\forall x)(\varphi\rightarrow\xi)\rightarrow((\exists x)\varphi\rightarrow\xi)$, where $x$ is not free in $\xi$.
 
\end{description}

The deduction rules of $L\forall^m$ are those of $L$ and the rule of generalization: from $\varphi$ infer $(\forall x)\varphi$. The definitions of proof and provability are analogous to the classical ones. We denote by $\Phi\vdash_{L\forall^m}\varphi$ the fact that $\varphi$ is provable in $L\forall^m$ from the set of formulas $\Phi$. For the sake of clarity, when it is clear from the context we will write $\vdash$ to refer to $\vdash_{L\forall^m}$. A set of formulas $\Phi$ is \emph{consistent} if $\Phi\not\vdash\overline{0}$.
 \end{defi}

\begin{defi} [\textbf{Semantics of Predicate Fuzzy Logics}] \label{evaluation} Consider a predicate language $\mathcal{P}=\langle Pred_{\mathcal{P}}, Func_{\mathcal{P}}, Ar_{\mathcal{P}} \rangle$ and let \textbf{A} be an $L$-algebra. We define an $\textbf{A}$\emph{-structure} $\mathrm{\mathbf{M}}$ for $\mathcal{P}$ as the triple $\langle M, (P_M)_{P\in Pred}, (F_M)_{F\in Func} \rangle$, where $M$ is a nonempty domain, $P_{\mathrm{\mathbf{M}}}$ is an $n$-ary fuzzy relation for each $n$-ary predicate symbol, i.e., a function from $M^n$ to $\textbf{A}$, identified with an element of $\textbf{A}$ if $n=0$; and $F_{\mathrm{\mathbf{M}}}$ is a function from $M^n$ to $M$, identified with an element of $M$ if $n=0$. As usual, if $\mathrm{\mathbf{M}}$ is an $\textbf{A}$-structure for $\mathcal{P}$, an $\mathrm{\mathbf{M}}$-evaluation of the object variables is a mapping $v$ assigning to each object variable an element of $M$. The set of all object variables is denoted by $Var$. If $v$ is an $\mathrm{\mathbf{M}}$-evaluation, $x$ is an object variable and $a\in M$, we denote by $v[x\mapsto a]$ the $\mathrm{\mathbf{M}}$-evaluation so that $v[x\mapsto a](x)=a$ and $v[x\mapsto a](y)=v(y)$ for $y$ an object variable such that $y\not=x$. If $\mathrm{\mathbf{M}}$ is an $\textbf{A}$-structure and $v$ is an $\mathrm{\mathbf{M}}$-evaluation, we define the \emph{values} of terms and the \emph{truth values} of formulas in $M$ for an evaluation $v$ recursively as follows:
 
\begin{description}

\item $||x||^{\small{\textbf{A}}}_{\mathrm{\mathbf{M}},v}=v(x)$;

\item $||F(t_1,\ldots,t_n)||^{\small{\textbf{A}}}_{\mathrm{\mathbf{M}},v}=F_{\mathrm{\mathbf{M}}}(||t_1||^{\small{\textbf{A}}}_{\mathrm{\mathbf{M}},v},\ldots,||t_n||^{\small{\textbf{A}}}_{\mathrm{\mathbf{M}},v})$, for $F\in Func$;

\item $||P(t_1,\ldots,t_n)||^{\small{\textbf{A}}}_{\mathrm{\mathbf{M}},v}=P_{\mathrm{\mathbf{M}}}(||t_1||^{\small{\textbf{A}}}_{\mathrm{\mathbf{M}},v},\ldots,||t_n||^{\small{\textbf{A}}}_{\mathrm{\mathbf{M}},v})$, for $P\in Pred$;

\item $||c(\varphi_1,\ldots,\varphi_n)||^{\small{\textbf{A}}}_{\mathrm{\mathbf{M}},v}=c_{\textbf{A}}(||\varphi_1||^{\small{\textbf{A}}}_{\mathrm{\mathbf{M}},v},\ldots,||\varphi_n||^{\small{\textbf{A}}}_{\mathrm{\mathbf{M}},v})$, for $c\in\mathcal{L}$;

\item $||(\forall x)\varphi||^{\small{\textbf{A}}}_{\mathrm{\mathbf{M}},v}=inf\{||\varphi||^{\small{\textbf{A}}}_{\mathrm{\mathbf{M}},v[x\rightarrow a]}\mid a\in M\}$;

\item $||(\exists x)\varphi||^{\small{\textbf{A}}}_{\mathrm{\mathbf{M}},v}=sup\{||\varphi||^{\small{\textbf{A}}}_{\mathrm{\mathbf{M}},v[x\rightarrow a]}\mid a\in M\}$.

\end{description}
If the infimum or the supremum do not exist, we take the truth value of the formula as undefined. We say that an $\textbf{A}$-structure is \emph{safe} if $||\varphi||^{\small{\textbf{A}}}_{\mathrm{\mathbf{M}},v}$ is defined for each formula $\varphi$ and each $\mathrm{\mathbf{M}}$-evaluation $v$. \end{defi}

 \noindent For a set of formulas $\Phi$, we write $||\Phi||^{\emph{\textbf{A}}}_{\mathrm{\mathbf{M}},v}=1$ if $||\varphi||^{\emph{\textbf{A}}}_{\mathrm{\mathbf{M}},v}=1$ for every $\varphi\in\Phi$. We say that $\langle\emph{\textbf{A}},\mathrm{\mathbf{M}}\rangle$ is a \emph{model of a set of formulas $\Phi$} if $||\varphi ||^{\emph{\textbf{A}}}_{\mathrm{\mathbf{M}},v}=1$ for any $\varphi\in\Phi$ and any \textbf{M}-evaluation $v$. We denote by $||\varphi||^{\emph{\textbf{A}}}_{\textbf{M}}=1$ that $||\varphi||^{\emph{\textbf{A}}}_{\textbf{M},v}=1$ for all \textbf{M}-evaluation $v$. We say that a formula $\varphi$ is \emph{satisfiable} if there exists a structure $\langle\emph{\textbf{{A}}},\textbf{M}\rangle$ such that $||\varphi||^{\emph{\textbf{A}}}_{\textbf{M}}=1$. In such case, we also say that $\varphi$ is \emph{satisfied by} $\langle\emph{\textbf{{A}}}, \textbf{M}\rangle$ or that $\langle\emph{\textbf{{A}}},\textbf{M}\rangle$ \emph{satisfies $\varphi$}.  Unless otherwise stated, from now on \emph{\textbf{A}} denotes an MTL-algebra and we refer to \emph{\textbf{A}}-structures simply as \emph{structures}. \smallskip

Now we recall the notion of homomorphism between fuzzy structures.

\begin{defi} {\em \textbf{\cite[Definition 6]{DeGaNo14}}} \label{def:mapping structures}
$\space$ Let $\langle\textbf{A},\mathrm{\mathbf{M}}\rangle$ and $\langle\textbf{B},\mathrm{\mathbf{N}}\rangle$ be structures, $f$ be a mapping from $\textbf{A}$ to $\textbf{B}$ and $g$ be a mapping from $M$ to $N$. The pair $\langle f,g\rangle$ is said to be a \emph{homomorphism} from $\langle\textbf{A},\mathrm{\mathbf{M}}\rangle$ to $\langle\textbf{B},\mathrm{\mathbf{N}}\rangle$ if $f$ is a homomorphism of ${L}$-algebras and for every $n$-ary function symbol $F$ and $d_1,\ldots,d_n\in M$,
$$g(F_{\mathrm{\mathbf{M}}}(d_1,\ldots,d_n))=F_{\mathrm{\mathbf{N}}}(g(d_1),\ldots,g(d_n)) $$

\noindent and for every $n$-ary predicate symbol $P$ and $d_1,\ldots,d_n\in M$,  
$$  \text{ \emph{(*) }} \text{If }P_{\mathrm{\mathbf{M}}}(d_1,\ldots,d_n)=1 \text{, then } P_{\mathrm{\mathbf{N}}}(g(d_1),\ldots,g(d_n))=1.$$
We say that a homomorphism $\langle f,g\rangle$ is \emph{strict} if instead of \emph{(*)} it satisfies the stronger condition: for every $n$-ary predicate symbol $P$ and $d_1,\ldots,d_n\in M$,  
$$P_{\mathrm{\mathbf{M}}}(d_1,\ldots,d_n)=1 \text{ if and only if } P_{\mathrm{\mathbf{N}}}(g(d_1),\ldots,g(d_n))=1.$$

\noindent Moreover we say that $\langle f,g\rangle$ is an \emph{embedding} if it is a strict homomorphism and both functions $f$ and $g$ are injective. And we say that an embedding $\langle f,g\rangle$ is an \emph{isomorphism} if both functions $f$ and $g$ are surjective.

\end{defi}  

\section{Horn clauses}  
\label{Horn clauses}

In this section we present a definition of Horn clause over predicate fuzzy logics that extends the classical definition in a natural way. In classical predicate logic, a \emph{basic Horn formula} is a formula of the form $ \alpha_{1}\wedge\dotsb \wedge\alpha_{n}\rightarrow\beta$, where $n\in\mathbb{N}$ and $\alpha_1,\ldots,\alpha_n,\beta$ are atomic formulas. Now we extend these definitions to work with predicate fuzzy logics. Observe that there is not a unique way to extend them due to the fact that, in predicate fuzzy logic, we have different conjunctions and implications. 

\begin{defi}[Basic Horn Formula]\label{strong basic}A \emph{basic Horn formula} is a formula of the form \begin{equation} \label{1}
 \alpha_1\&\dotsb\&\alpha_n\rightarrow\beta \hfill 
\end{equation} 
where $n\in\mathbb{N}$, $\alpha_1,\ldots,\alpha_n, \beta$ are atomic formulas. 
\end{defi} 

The formula obtained by substitution in expression (\ref{1}) of the strong conjunction $\&$ by the weak conjunction $\wedge$  will be called \emph{basic weak Horn formula}. From now on, for the sake of clarity, we will refer to the basic weak Horn formulas as \emph{basic w-Horn formulas}. 

Analogously to classical logic, disjunctive definitions of basic Horn formulas can be defined. Nevertheless, it is an easy exercise to check that, for predicate fuzzy logics, these disjunctive forms are not in general equivalent to the implicational ones that we have introduced here. Here we focus our analysis on the implicational Horn clauses and we leave for future work the study of the properties of disjunctive Horn clauses.

\begin{defi}
\label{qf Horn} A \emph{quantifier-free Horn formula} is a formula of the form \newline $\phi_1\&\dotsb\&\phi_m$ where $m\in\mathbb{N}$ and $\phi_i$ is a basic Horn formula for every $1\leq i\leq m$. If $\phi_i$ is a basic w-Horn formula for every $1\leq i\leq m$, we say that $\phi_1\wedge\dotsb\wedge\phi_m$ is a \emph{quantifier-free w-Horn formula}.
\end{defi}
From now on, whenever it is possible, we present a unique definition for both the strong and the weak version, we use the \emph{w-} symbol into parenthesis.

\begin{defi}\label{Horn}A \emph{(w-)Horn formula} is a formula of the form $Q\gamma$, where $Q$ is a (possibly empty) string of quantifiers $(\forall x),(\exists x)$... and $\gamma$ is a quantifier-free (w-)Horn formula. A \emph{(w-)Horn clause} (or \emph{universal (w-)Horn formula}) is a (w-)Horn formula in which the quantifier prefix (if any) has only universal quantifiers. A \emph{(w-)universal Horn theory} is a set of (w-)Horn clauses. 
\end{defi}

Observe that, in classical logic, the formula $(\forall x)\varphi \wedge (\forall x) \psi$ is logically equivalent to $(\forall x)(\varphi \wedge \psi).$ This result can be used to prove that every Horn clause is equivalent in classical logic to a conjunction of formulas of the form $(\forall x_1) \ldots (\forall x_k)\varphi$, where $\varphi$ is a basic Horn formula. Having in mind these equivalences, it is easy to see that the set of all Horn clauses is recursively defined in classical logic by the following rules: \begin{itemize}

\item[1.] If $\varphi$  is a basic Horn formula, then $\varphi$ is a Horn clause;

\item[2.] If $\varphi$ and $\psi$ are Horn clauses, then $\varphi\wedge\psi$ is a Horn clause;

\item[3.]  If $\varphi$  is a Horn clause, then $(\forall x)\varphi$ is a Horn clause.

\end{itemize}

In MTL$\forall^m$ we can deduce $(\forall x)\varphi \wedge (\forall x) \psi\leftrightarrow(\forall x)(\varphi \wedge \psi)$. This fact allows us to show that in MTL$\forall^m$, any w-Horn clause is equivalent to a weak conjunction of formulas of the form $(\forall x_1)\dotsb(\forall x_k)(\varphi)$ where $\varphi$ is a basic w-Horn formula. Thus, w-Horn clauses can be recursively defined in MTL$\forall^m$ as above. But it is not the case for the strong conjunction since $(\forall x)\varphi \& (\forall x) \psi\leftrightarrow(\forall x)(\varphi \& \psi)$ can not be deduced from MTL$\forall^m$ (we refer to \cite[Remark p.281]{EsGo01}). So the set of Horn clauses is not recursively defined in MTL$\forall^m$.

\section{Term structures associated to a set of formulas} 
\label{Term structure associated to a set of formulas}

In this section we introduce the notion of term structure associated to a set of formulas over predicate fuzzy logics. We study the particular case of sets of universal Horn formulas and prove that the term structure associated to these sets of formulas is free. Term structures have been used in classical logic to prove the satisfiability of a set of consistent sentences, see for instance \cite[Ch.5]{EbiFlu94}. From now on we assume that we work in a language with a binary predicate symbol $\approx$ interpreted as a similarity. We assume also that the axiomatization of the logic $L\forall^m$ contains the following axioms for $\approx$.

\begin{defi} $\emph{\textbf{ \cite[Definitions 5.6.1, 5.6.5]{Ha98}}}$ \label{def similarity} 
Let $\approx$ be a binary predicate symbol, the following are the axioms of similarity and congruence: 
\begin{itemize}

\item[S1.] $(\forall x)x\approx x$ 

\item[S2.] $(\forall x)(\forall y)(x\approx y\rightarrow y\approx x$) 

\item[S3.] $(\forall x)(\forall y)(\forall z)(x\approx y \& y\approx z\rightarrow x\approx z)$  \end{itemize} 
\begin{itemize}
\item[C1.]  For each $n$-ary function symbol  $F$, \end{itemize} {\footnotesize
$(\forall x_1)\dotsb(\forall x_n)(\forall y_1)\dotsb(\forall y_n)(x_1\approx y_1\&\dotsb \& x_n\approx y_n\rightarrow F(x_1,\ldots,x_n)\approx F(y_1,\ldots,y_n))$
}
\begin{itemize}
\item[C2.]  For each $n$-ary predicate symbol  $P$, \end{itemize} {\footnotesize 
$(\forall x_1)\dotsb(\forall x_n)(\forall y_1)\dotsb(\forall y_n)(x_1\approx y_1\&\dotsb \& x_n\approx y_n\rightarrow (P(x_1, \ldots, x_n)\leftrightarrow P(y_1,\ldots, y_n)))$ }
\end{defi}

\begin{defi}\label{relacio}   
Let $\Phi$ be a set of formulas, we define a binary relation on the set of terms, denoted by $\sim$, in the following way: for every terms $t_1,t_2$,
\begin{center}
$t_1\sim t_2$ if and only if $\Phi\vdash t_1\approx t_2$. 
\end{center}
\end{defi}

By using \cite[Prop.1(5)]{EsGo01}, it is easy to check that for every set of formulas $\Phi$, $\sim$ is an equivalence relation. From now on we denote by $\overline{t}$ the $\sim$-class of the term $t$. The next result, which states that $\sim$ is compatible with the symbols of the language, can be easily proven using the Axioms of Congruence of Definition \ref{def similarity}.

\begin{lemma} \label{f} Let $\Phi$ be a set of formulas. The relation $\sim$ has the following property: if for every $1\leq i\leq n$, $t_i\sim t'_i$, then 

  \begin{itemize}
 \item[(i)] For any $n$-ary function symbol $F$, $F(t_1,\ldots,t_n)\sim F(t'_1,\ldots,t'_n)$,  

\item[(ii)] For any $n$-ary predicate symbol $P$, \small{
$\Phi\vdash P(t_1,\ldots, t_n)$ \text{iff} $\Phi\vdash P(t'_1, \ldots, t'_n)$}
\end{itemize}
\end{lemma}

\smallskip

\begin{defi}  [Term Structure] \label{structure}
Let $\Phi$ be a consistent set of formulas. We define the following structure $\langle\textbf{B},\mathrm{\mathbf{T}}^{\Phi}\rangle$, where $\textbf{B}$ is the two-valued Boolean algebra, $\mathrm{\mathbf{T}}^{\Phi}$ is the set of all equivalence classes of the relation $\sim$ and 
\begin{itemize}

\item For any $n$-ary function symbol $F$, 
$$F_{\mathrm{\mathbf{T}}^{\Phi}}(\overline{t}_1,\ldots,\overline{t}_n)=\overline{F(t_1,\ldots,t_n)}$$

\item For any $n$-ary predicate symbol $P$, 
 $$ P_{\mathrm{\mathbf{T}}^{\Phi}}(\overline{t}_1,\ldots,\overline{t}_n)=\begin{cases} 1, & \mbox{if } \Phi\vdash P(t_1,\ldots, t_n) \\ 0, & \mbox{otherwise } \end{cases} $$ 

\end{itemize}
We call $\langle\textbf{B},\mathrm{\mathbf{T}}^{\Phi}\rangle$ the \emph{term structure associated to $\Phi$}. 
\end{defi}

Notice that for every $0$-ary function symbol $c$, $c_{\mathrm{\mathbf{T}}^{\Phi}}=\overline{c}$. By using Lemma \ref{f}, it is easy to prove that the structure $\langle\emph{\textbf{B}},\mathrm{\mathbf{T}}^{\Phi}\rangle$ is well-defined, because the conditions are independent from the choice of the representatives. Observe that, so defined, $\langle\emph{\textbf{B}},\mathrm{\mathbf{T}}^{\Phi}\rangle$ is a classical structure. The following lemma agrees with this classical character. 
\begin{lemma} \label{crisp}
Let $\Phi$ be a consistent set of formulas. The interpretation of the $\approx$ symbol in the structure $\langle\textbf{B},\mathrm{\mathbf{T}}^{\Phi}\rangle$ is the crisp equality. 
\end{lemma}
\begin{proof}
Let $t_1,t_2$ be terms. We have $\overline{t}_1=\overline{t}_2$ iff $t_1\sim t_2$ iff $\Phi\vdash t_1\approx t_2$ iff 
$\overline{t_1}\approx_{\mathrm{\mathbf{T}}^{\Phi}} \overline{t_2}$ (this last step by Definition \ref{structure}). \end{proof}

\bigskip

Now we prove some technical lemmas that will allow us to show that the term structrure $\langle\emph{\textbf{B}},\mathrm{\mathbf{T}}^{\Phi}\rangle$ is free.

\begin{defi} \label{canonical}
Given a consistent set of formulas $\Phi$, let $e^{\Phi}$ be the following $\emph{\textbf{T}}^{\Phi}$-evaluation: $e^{\Phi}(x)=\overline{x}$. We call $e^{\Phi}$ the \emph{canonical evaluation of} $\langle\textbf{B},\mathrm{\mathbf{T}}^{\Phi}\rangle$. 
\end{defi}

\begin{lemma} \label{terms and atomic formulas}
 Let $\Phi$ be a consistent set of formulas, the following holds:
 \begin{itemize}
\item[(i)] For any term $t$, $|| t ||^{\textbf{B}}_{\mathrm{\mathbf{T}}^{\Phi},e^{\Phi}}=\overline{t}$.

\item[(ii)] For any atomic formula $\varphi$, $|| \varphi||^{\textbf{B}}_{\mathrm{\mathbf{T}}^{\Phi},e^{\Phi}}=1$ if and only if $\Phi\vdash\varphi$.

\item[(iii)] For any atomic formula $\varphi$, $|| \varphi||^{\textbf{B}}_{\mathrm{\mathbf{T}}^{\Phi},e^{\Phi}}=0$ if and only if $\Phi\not \vdash\varphi$.
\end{itemize}
\end{lemma}

\begin{proof}
(i) By induction on the complexity of $t$ and Definitions \ref{structure} and \ref{canonical}. \newline
(ii) Let $P$ be an $n$-ary predicate symbol and $t_1,\ldots,t_n$ be terms, we have:

\medskip$\begin{array}{rr} ||P(t_1,\ldots,t_n)||^{\emph{\textbf{B}}}_{\mathrm{\mathbf{T}}^{\Phi},e^{\Phi}}=1 & \text{iff} 
\\[2ex]  P_{\mathrm{\mathbf{T}}^{\Phi}}(|| t_1 ||^{\emph{\textbf{B}}}_{\mathrm{\mathbf{T}}^{\Phi},e^{\Phi}},\ldots,|| t_n ||^{\emph{\textbf{B}}}_{\mathrm{\mathbf{T}}^{\Phi},e^{\Phi}})=1 & \text{iff}

\\[2ex]  P_{\mathrm{\mathbf{T}}^{\Phi}}(\overline{t}_1,\ldots,\overline{t}_n)=1 & \text{iff}
\\[2ex] \Phi\vdash P(t_1,\ldots, t_n) & \end{array}$

\medskip \noindent The second equivalence is by (i) of the present Lemma, and the third one by Definition \ref{structure}. (iii) holds because $\langle\emph{\textbf{B}},\mathrm{\mathbf{T}}^{\Phi}\rangle$ is a classical structure. \end{proof}

\bigskip
Observe that, since terms are the smallest significance components of a first-order language, Lemma \ref{terms and atomic formulas} (ii) and (iii) can be read as saying that term structures are minimal with respect to atomic formulas. Intuitively speaking, the term structure picks up the positive atomic information associated to $\Phi$. 

\smallskip

\begin{lemma} \label{generates} Let $\Phi$ be a consistent set of formulas. The set $\{\overline{x}\mid x\in Var\}$ generates the universe $T^{\Phi}$ of the term structure associated to $\Phi$.  
\end{lemma}
\begin{proof}
Let $\overline{t(x_1,\ldots,x_n)}\in T^{\Phi}$. By Lemma \ref{terms and atomic formulas}, $$\overline{t(x_1,\ldots,x_n)}=||t(x_1,\ldots,x_n) ||^{\emph{\textbf{B}}}_{\mathrm{\mathbf{T}}^{\Phi},e^{\Phi}}$$ and by the semantics of predicate fuzzy logics (Definition \ref{evaluation}), \begin{center}

$||t(x_1,\ldots,x_n) ||^{\emph{\textbf{B}}}_{\mathrm{\mathbf{T}}^{\Phi},e^{\Phi}}=t_{T^{\Phi}}(||x_1||^{\emph{\textbf{B}}}_{\mathrm{\mathbf{T}}^{\Phi},e^{\Phi}},\ldots,||x_n ||^{\emph{\textbf{B}}}_{\mathrm{\mathbf{T}}^{\Phi},e^{\Phi}})=t_{\mathrm{\mathbf{T}}^{\Phi}}(\overline{x}_1,\ldots,\overline{x}_n)$. \end{center}  \end{proof} 

\bigskip

Term structures do not necessarily satisfy the theory to which they are associated. In classical logic, if it is the case, from an algebraic point of view, the minimality of the term structure is revealed by the fact that the structure is \emph{free}. A model of a theory is free if there is a unique homomorphism from this model to any other model of the theory. Free structures have their origin in category theory, as a generalization of free groups (for a definition of free structure in category theory, see \cite[Def. 4.7.17]{BaWe98}). Free structures are also named \emph{initial} in \cite[Def. 2.1 (i)]{Mak87}. In the context of computer science, they appeared for the first time in \cite{GoThWaWr75}. 

The possibility given by fuzzy logic of defining the term structure associated to a theory using the similarity symbol $\approx$ leads us to a notion of free structure restricted to the class of reduced models of that theory, as we will prove in next theorem. Remember that \emph{reduced structures} are those whose Leibniz congruence is the identity. By \cite[Lemma 20]{De12}, a structure $\langle\emph{\textbf{A}},\mathrm{\mathbf{M}}\rangle$ is reduced iff it has the \emph{equality property} (EQP) (that is, for any $d,e\in M$, 
$d\approx_{\mathrm{\mathbf{M}}}  e$ iff $d=e$). 

\begin{theorem} \label{initial model}
Let $\Phi$ be a consistent set of formulas with $|| \Phi||^{\textbf{B}}_{\mathrm{\mathbf{T}}^{\Phi},e^{\Phi}}=1$. Then, $\langle\textbf{B},\mathrm{\mathbf{T}}^{\Phi}\rangle$ is a free structure in the class of the reduced models of $\Phi$, i.e., for every reduced structure $\langle\textbf{A},\mathrm{\mathbf{M}}\rangle$ and every evaluation $v$ such that $|| \Phi||^{\textbf{A}}_{\mathrm{\mathbf{M}},v}=1$, there is a unique homomorphism $\langle f,g\rangle$ from $\langle\textbf{B},\mathrm{\mathbf{T}}^{\Phi}\rangle$ to $\langle\textbf{A},\mathrm{\mathbf{M}}\rangle$ such that for every $x \in Var$, $g(\overline{x})=v(x)$.
 \end{theorem}

\begin{proof} Let $\langle\emph{\textbf{A}},\mathrm{\mathbf{M}}\rangle$ be a reduced structure and $v$ an $\textbf{M}$-evaluation such that $|| \Phi||^{\emph{\textbf{A}}}_{\mathrm{\mathbf{M}},v}=1$. Now let $f:\emph{\textbf{B}}\rightarrow\emph{\textbf{A}}$ be the identity and define $g$ by: $g(\overline{t})=|| t ||^{\emph{\textbf{A}}}_{\mathrm{\mathbf{M}},v}$ for every term $t$. We show that $\langle f,g\rangle$ is the desired homomorphism (for the definition of homomorphism see the Preliminaries section, Definition \ref{def:mapping structures}).

First let us check that $g$ is well-defined. Given terms $t_1,t_2$ with $\overline{t}_1=\overline{t}_2$, that is, $t_1\sim t_2$, by Definition \ref{relacio}, $\Phi\vdash t_1\approx t_2$. Then, since $|| \Phi||^{\emph{\textbf{A}}}_{\mathrm{\mathbf{M}},v}=1$, we have $|| t_1\approx t_2||^{\emph{\textbf{A}}}_{\mathrm{\mathbf{M}},v}=1$. But $\langle\emph{\textbf{A}},\mathrm{\mathbf{M}}\rangle$ is reduced, which by \cite[Lemma 20]{De12} is equivalent to have the EQP; therefore $|| t_1 ||^{\emph{\textbf{A}}}_{\mathrm{\mathbf{M}},v}=|| t_2||^{\emph{\textbf{A}}}_{\mathrm{\mathbf{M}},v}$, that is, $g(\overline{t_1})=g(\overline{t_2})$.

Now, let us see that $g$ is a homomorphism. Let $\overline{t}_1,\ldots,\overline{t}_n\in T^{\Phi}$ be terms and $F$ be an $n$-ary function symbol. By Definition \ref{structure}, we have that $$F_{\mathrm{\mathbf{T}}^{\Phi}}(\overline{t}_1,\ldots,\overline{t}_n)=\overline{F(t_1,\ldots,t_n)}$$ and then
$g(F_{\mathrm{\mathbf{T}}^{\Phi}}(\overline{t}_1,\ldots,\overline{t}_n))=g(\overline{F(t_1,\ldots,t_n)})=|| F(t_1,\ldots,t_n) ||^{\emph{\textbf{A}}}_{\mathrm{\mathbf{M}},v}= \newline =F_{\textbf{M}}(|| t_1||^{\emph{\textbf{A}}}_{\mathrm{\mathbf{M}},v},\ldots ,|| t_n ||^{\emph{\textbf{A}}}_{\mathrm{\mathbf{M}},v})= F_{\textbf{M}}(g(\overline{t}_1),\ldots,g(\overline{t}_n))$. 

\bigskip

Let $P$ be an $n$-ary predicate symbol such that  $P_{\mathrm{\mathbf{T}}^{\Phi}}(\overline{t}_1,\ldots,\overline{t}_n)=1$. By Definition \ref{structure}, $\Phi\vdash P(t_1,\ldots,t_n)$. Since $|| \Phi||^{\emph{\textbf{A}}}_{\mathrm{\mathbf{M}},v}=1$, we have $$|| P(t_1,\ldots,t_n)||^{\emph{\textbf{A}}}_{\mathrm{\mathbf{M}},v}=1$$ and then $P_{\mathrm{\mathbf{M}}}(|| t_1||^{\emph{\textbf{A}}}_{\mathrm{\mathbf{M}},v},\ldots ,|| t_n ||^{\emph{\textbf{A}}}_{\mathrm{\mathbf{M}},v})=1$, that is, $P_{\mathrm{\mathbf{M}}}(g(\overline{t}_1),\ldots,g(\overline{t}_n))=1$.

\medskip

Finally, since by Lemma \ref{generates} the set $\{\overline{x}\mid x\in Var\}$ generates the universe $T^{\Phi}$ of the term structure associated to $\Phi$, $\langle f,g\rangle$ is the unique homomorphism such that for every $x \in Var$, $g(\overline{x})=v(x)$.   \end{proof}

\bigskip 

Observe that in languages in which the similarity symbol is interpreted by the crisp identity, by using an analogous argument to the one in Theorem \ref{initial model}, we obtain that the term structure is free in all the models of the theory and not only in the class of reduced models. 

\bigskip

To end this section we prove that the term structure associated to a universal Horn theory is a model of this theory. We have shown above in Section \ref{Horn clauses} that the set of Horn clauses is not recursively defined in MTL$\forall^m$. For that reason we will present here proofs that differ from the proofs of the corresponding results in classical logic, using induction on the rank of a formula instead of induction on the set of the (w-)Horn clauses. We introduce first the notion of \emph{rank of a formula} $\varphi$. Our definition is a variant of the notion of \emph{syntactic degree of a formula} in 
\cite[Definition 5.6.7]{Ha98}).
\begin{description} 
\item $rk(\varphi)=0$, if $\varphi$ is atomic;
\item $rk(\neg\varphi)=rk((\exists x)\varphi)=rk((\forall x)\varphi)=rk(\varphi)+1$;
\item $rk(\varphi\circ\psi)=rk(\varphi)+rk(\psi)$, for every binary propositional connective $\circ$.
\end{description}

\begin{lemma} \label{Horn substitucio}
Let $\varphi$ be a (w-)Horn clause where $x_1,\ldots,x_m$ are pairwise distinct free variables. Then, for every terms $t_1,\ldots,t_m$, $$\varphi (t_1,\ldots,t_m/x_1,\ldots,x_m)$$ is a (w-)Horn clause.
\end{lemma}
\begin{proof}
We prove it for the strong conjunction but the proof is analogous for the weak conjunction. By induction on $rk(\varphi)$.

\bigskip

\underline{Case $rk(\varphi)=0$}. If $\varphi$ is a basic Horn formula of the form $\psi_1\& \ldots \&\psi_n\rightarrow\psi$, it is clear that $\varphi (t_1,\ldots,t_m/x_1,\ldots,x_m)$ is still a basic Horn formula. In case that $\varphi=\phi_1\& \dotsb\&\phi_l$ is a conjunction of basic Horn formulas, note that $\varphi (t_1,\ldots,t_m/x_1,\ldots,x_m)$ has the same form as $\varphi$.

\bigskip

\underline{Case $rk(\varphi)=n+1$}. Assume inductively that for any Horn clause $\psi$ where $x_1,\ldots,x_m$ are pairwise distinct free variables in $\psi$ and whose rank is $n$, the formula $\psi (t_1,\ldots,t_m/x_1,\ldots,x_m)$ is a Horn clause. Let $\varphi$ be a Horn clause of rank $n+1$, then $\varphi$ is of the form $(\forall y)\psi$, where $\psi$ has rank $n$. Assume without loss of generality that and $y\not\in\{x_1,\ldots,x_m\}$, then$$[(\forall y)\psi](t_1,\ldots,t_m/x_1,\ldots,x_m)=(\forall y)[\psi(t_1,\ldots,t_m/x_1,\ldots,x_m)]$$ thus we can apply the inductive hypothesis to obtain the desired result.  \end{proof}

\begin{theorem} \label{theorem Horn formulas}
Let $\Phi$ be a consistent set of formulas. For every (w-)Horn clause $\varphi$, if $\Phi\vdash\varphi$, then $|| \varphi||^{\textbf{B}}_{\mathrm{\mathbf{T}}^{\Phi},e^{\Phi}}=1$.
\end{theorem}

\begin{proof} We prove it for the strong conjunction but the proof is analogous for the weak conjunction. By induction on $rk(\varphi)$.

\bigskip

\underline{Case $rk(\varphi)=0$.} We can distinguish two subcases:  

\medskip

1) If $\varphi=\psi_1\&\dotsb\&\psi_n\rightarrow\psi$ is a basic Horn formula, we have to show that $||\psi_1\&\dotsb\&\psi_n||^{\emph{\textbf{B}}}_{\mathrm{\mathbf{T}}^{\Phi},e^{\Phi}}\leq|| \psi||^{\emph{\textbf{B}}}_{\mathrm{\mathbf{T}}^{\Phi},e^{\Phi}}$. If $|| \psi||^{\emph{\textbf{B}}}_{\mathrm{\mathbf{T}}^{\Phi},e^{\Phi}}=1$, we are done. Otherwise, by Definition \ref{structure}, $\Phi\not \vdash \psi$. Consequently, since $\Phi\vdash \psi_1\&\dotsb\&\psi_n\rightarrow\psi$, $\Phi\not \vdash \psi_1\&\dotsb\&\psi_n$ and thus for some $1 \leq i \leq n$, $\Phi\not \vdash \psi_i$. By Lemma \ref{terms and atomic formulas} (ii), we have $|| \psi_i||^{\emph{\textbf{B}}}_{\mathrm{\mathbf{T}}^{\Phi},e^{\Phi}}=0$ and then $||\psi_1\&\dotsb\&\psi_n||^{\emph{\textbf{B}}}_{\mathrm{\mathbf{T}}^{\Phi},e^{\Phi}}=0$. Therefore, we can conclude $||\psi_1\&\dotsb\&\psi_n||^{\emph{\textbf{B}}}_{\mathrm{\mathbf{T}}^{\Phi},e^{\Phi}}\leq|| \psi||^{\emph{\textbf{B}}}_{\mathrm{\mathbf{T}}^{\Phi},e^{\Phi}}$. Note that if $n=0$, $\varphi$ is an atomic formula and the property holds by Lemma \ref{terms and atomic formulas} (ii).

\smallskip

2) If $\varphi=\psi_1\&\dotsb\&\psi_n$ is a conjunction of basic Horn formulas and $\Phi\vdash\varphi$, then for every $1 \leq i \leq n$, $\Phi\vdash \psi_i$. Thus, by 1), for every $1 \leq i \leq n$, $||\psi_i||^{\emph{\textbf{B}}}_{\mathrm{\mathbf{T}}^{\Phi},e^{\Phi}}=1$ and then $|| \varphi||^{\emph{\textbf{B}}}_{\mathrm{\mathbf{T}}^{\Phi},e^{\Phi}}=1$.

\bigskip

\underline{Case $rk(\varphi)=n+1$.}

\medskip
 If $\varphi=(\forall x)\phi(x)$ is a Horn clause, where $rk(\phi(x))=n$ and $\Phi\vdash\varphi$, by Axiom \emph{$\forall 1$} of $L\forall^m$, for every term $t$, $\Phi\vdash\phi(t/x)$. Since by Lemma \ref{Horn substitucio}, $\phi(t/x)$ is also a Horn clause and $rk(\phi(t/x))=n$, we can apply the inductive hypothesis and hence for every term $t$, $||\phi(t/x)||^{\emph{\textbf{B}}}_{\mathrm{\mathbf{T}}^{\Phi},e^{\Phi}}=1$, that is, by Lemma \ref{terms and atomic formulas} (i), for every element $\overline{t}$ of the domain, $||\phi(x)||^{\emph{\textbf{B}}}_{\mathrm{\mathbf{T}}^{\Phi},e^{\Phi}(x\rightarrow\overline{t})}=1$. Therefore, we can conclude that $||(\forall x)\phi(x)||^{\emph{\textbf{B}}}_{\mathrm{\mathbf{T}}^{\Phi},e^{\Phi}}=1$.\end{proof}
 
\bigskip

Observe that the inverse direction of Theorem \ref{theorem Horn formulas} is not true. Assume that we work in G\"odel predicate fuzzy logic G$\forall$. Let $P$ be a $1$-ary predicate symbol, $\overline{c}$ be an individual constant, $\Phi=\{\neg(P\overline{c}\rightarrow\overline{0})\}$ and $\varphi=P\overline{c}\rightarrow\overline{0}$. Now we show that $|| \varphi||^{\emph{\textbf{B}}}_{\mathrm{\mathbf{T}}^{\Phi}}=1$, but $\Phi\not \vdash\varphi$. First, in order to show that $\Phi\not \vdash\varphi$, consider a G-algebra $\emph{\textbf{A}}$ with domain the real interval $[0,1]$ and a structure $\langle\emph{\textbf{A}},\mathrm{\mathbf{M}}\rangle$ such that $||P\overline{c}||^{\emph{\textbf{A}}}_{\mathrm{\mathbf{M}}}=0.8$, then we have that $||\Phi||^{\emph{\textbf{A}}}_{\mathrm{\mathbf{M}}}=1$ and $||P\overline{c}\to \overline{0}||^{\emph{\textbf{A}}}_{\mathrm{\mathbf{M}}}\neq1$ consequently $\Phi\not\vdash_GP\overline{c}\to \overline{0}$. Using the same structure we obtain also that $\Phi\not\vdash_GP\overline{c}$. Finally, since $\Phi\not\vdash_GP\overline{c}$, by Lemma \ref{terms and atomic formulas}, $||P\overline{c}||^{\emph{\textbf{B}}}_{\mathrm{\mathbf{T}}^{\Phi}}=0$ and then $||\varphi||^{\emph{\textbf{B}}}_{\mathrm{\mathbf{T}}^{\Phi}}=1$.

\bigskip
Remark that, as a corollary of Theorem \ref{theorem Horn formulas}, we have that the substructure of $\langle\emph{\textbf{B}},\mathrm{\mathbf{T}}^{\Phi}\rangle$ generated by the set of ground terms is also a model for all universal Horn sentences that are consequences of the theory. Another important corollary of Theorem \ref{theorem Horn formulas} is the following:

\begin{corollary} \label{classic} Every consistent set of (w-)Horn clauses without free variables has a classical model.
\end{corollary}

Observe that Corollary \ref{classic} is not true in general. The consistent sentence $\neg (\overline{1} \to Pa) \& \neg (Pa \to \overline{0})$ has no classical model.

\section{Herbrand Structures}
\label{Herbrand Structures}

In this section we introduce Herbrand structures for fuzzy universal Horn theories. They are a prominent form of term structures, specially helpful when dealing with sets of equality-free formulas (that is, formulas in which the symbol $\approx$ does not occur), the reason is that, as it is shown below in Lemma \ref{equality}, no non-trivial equations are derivable from a set of equality-free formulas. In classical logic, Herbrand structures have been used to present a simplified version of a term structure associated to a consistent theory \cite[Ch.11]{EbiFlu94}, and they have also a relevant role in the foundation of logic programming (see for instance \cite{DoPo10}). Regarding Herbrand structures in fuzzy logic programming, we refer to the works \cite{Ge05,Voj01,Ebra01}. Throughout this section we assume that the symbol $\approx$ is interpreted always as the crisp identity and that there is at least an individual constant in the language.

\begin{lemma} \label{equality} Let $\Phi$ be a consistent set of equality-free formulas, then for every terms $t_1,t_2$,

$$ \text{If }\Phi\vdash t_1\approx t_2, \text{ then }t_1=t_2.$$

\end{lemma}
\begin{proof}
Assume that $\Phi$ is a consistent set of equality-free formulas and $\Phi\vdash t_1\approx t_2$ for terms $t_1,t_2$ of the language. Since CL$\forall$ is an extension of MTL$\forall^m$, $\Phi\vdash t_1\approx t_2$ in CL$\forall$. Then, by the analogous classical result \cite[Ch. 11, Th. 3.1]{EbiFlu94}, we have $t_1=t_2$. \end{proof}

\begin{defi} [Herbrand Structure] \label{Herbrand structure}
The \emph{Herbrand universe of a predicate language} is the set of all ground terms of the language. A \emph{Herbrand structure} is a structure $\langle\textbf{A},\emph{\textbf{H}}\rangle$, where $\emph{\textbf{H}}$ is the Herbrand universe, and: \begin{itemize} 

\item[] For any individual constant symbol $c$, $c_{\emph{\textbf{H}}}=c$. 

\item[] For any $n$-ary function symbol $F$ and any $t_1,\ldots,t_n\in H$, \begin{center}

\smallskip

$F_{\emph{\textbf{H}}}(t_1,\ldots,t_n)=F(t_1,\ldots,t_n)$

\end{center}

\end{itemize} 

\end{defi}

Observe that in Definition \ref{Herbrand structure} no restrictions are placed on the interpretations of the predicate symbols and on the algebra we work over. The canonical models $\langle\emph{\textbf{Lind}}_{T},\mathbf{CM}(T)\rangle$ introduced in {\cite[Def.9]{CiHa06} are examples of Herbrand structures. In these structures $\emph{\textbf{Lind}}_{T}$ is the Lindenbaum algebra of a theory $T$ and the domain of $\mathbf{CM}(T)$ is the set of individual constants (the language in \cite{CiHa06} do not contain function symbols). Now we introduce a particular case of Herbrand structure and we show that every consistent Horn clause without free variables has a model of this kind.

\begin{defi}  [H-structure and H-model] \label{Herbrand boolean} Let $\overline{H}$ be the set of all \newline
equality-free sentences of the form $P(t_1,\ldots,t_n)$, where $t_1,\ldots,t_n$ are ground terms, $n\geq 1$ and $P$ is an $n$-ary predicate symbol. For every subset $H$ of $\overline{H}$, we define the Herbrand structure $\langle\textbf{B},\mathbf{N}^{\emph{H}}\rangle$, where $\textbf{B}$ is the two-valued Boolean algebra, the domain $\mathbf{N}^{\emph{H}}$ is the set of all ground terms of the language, the interpretation of the function symbols is as in every Herbrand structure and the interpretation of the predicate symbols is as follows: for every $n\geq 1$ and every $n$-ary predicate symbol $P$,

$$ P_{\mathrm{\mathbf{N}}^{\emph{H}}}(t_1,\ldots,t_n)=\begin{cases} 1, & \mbox{if } P(t_1,\ldots,t_n) \in H \\ 0, & \mbox{otherwise. } \end{cases}  $$
We call this type of Herbrand structures \emph{H-structures}. If $\Phi$ is a set of sentences, we say that an \emph{H}-structure is an \emph{H-model} of $\Phi$ if it is a model of $\Phi$.

 \end{defi}

\begin{proposition} \label{proposition} Let $\langle\textbf{A},\mathbf{M}\rangle$ be a structure and H be the set of all atomic equality-free sentences $\sigma$ such that $||\sigma||^{\textbf{A}}_{\emph{\textbf{M}}}=1$. Then, for every equality-free sentence $\varphi$ which is a (w-)Horn clause,  if $||\varphi||^{\textbf{A}}_{\emph{\textbf{M}}}=1$, then $||\varphi||^{\textbf{B}}_{\emph{\textbf{N}}^{\emph{H}}}=1$, where $\langle\textbf{B},\emph{\textbf{N}}^{\emph{H}}\rangle$ is an \emph{H}-structure as in Definition \ref{Herbrand boolean}.
\end{proposition}

\begin{proof} We prove it for the strong conjunction but the proof is analogous for the weak conjunction. Assume that $\varphi$ is an equality-free sentence which is a Horn clause and $||\varphi||^{\emph{\textbf{A}}}_{\textbf{M}}=1$. We proceed by induction on the rank of $\varphi$ 

\bigskip

\underline{Case $rk(\varphi)=0$.} We distinguish two cases: 

\bigskip

1) If $\varphi=\psi_1\&\dotsb\&\psi_n\rightarrow\psi$ is a basic Horn formula, we have to show that $||\psi_1\&\dotsb\&\psi_n||^{\emph{\textbf{B}}}_{\textbf{N}^{\text{H}}} \leq ||\psi||^{\emph{\textbf{B}}}_{\textbf{N}^{\text{H}}}$. If $||\psi||^{\emph{\textbf{B}}}_{\textbf{N}^{\text{H}}}=1$, we are done. Otherwise, by Definition \ref{Herbrand boolean}, $\psi\not \in$ H , and thus $||\psi||^{\emph{\textbf{A}}}_{\textbf{M}}\not = 1$. Therefore, since $||\varphi||^{\emph{\textbf{A}}}_{\textbf{M}}=1$, we have that $||\psi_1\&\dotsb\&\psi_n||^{\emph{\textbf{A}}}_{\textbf{M}}\not = 1$. Consequently for some $1 \leq i \leq n$, $||\psi_i||^{\emph{\textbf{A}}}_{\textbf{M}}\not = 1$, therefore $\psi_i\not \in$ H and $||\psi_i||^{\emph{\textbf{B}}}_{\textbf{N}^{\text{H}}} =0$ and then $||\psi_1\&\dotsb\&\psi_n||^{\emph{\textbf{B}}}_{\textbf{N}^{\text{H}}} =0$. Hence, $||\psi_1\&\dotsb\&\psi_n||^{\emph{\textbf{B}}}_{\textbf{N}^{\text{H}}} \leq ||\psi||^{\emph{\textbf{B}}}_{\textbf{N}^{\text{H}}}$.

\medskip

2)  If $\varphi=\psi_1\&\dotsb\&\psi_n$ is a strong conjunction of basic Horn formulas, then by 1) we have that $||\psi_i||^{\emph{\textbf{A}}}_{\textbf{M}}=1$ implies $||\psi_i||^{\emph{\textbf{B}}}_{\textbf{N}^{\text{H}}}=1$, for each $i\in\{1,\ldots,n\}$. Thus,  if $||\varphi||^{\emph{\textbf{A}}}_{\textbf{M}}=1$, then $||\varphi||^{\emph{\textbf{B}}}_{\textbf{N}^{\text{H}}}=1$.  

\bigskip

\underline{Case $rk(\varphi)=n+1$.}

\bigskip

Let $\varphi=(\forall x)\phi(x)$ be a Horn clause with $rk(\phi(x))=n$. Since $||\varphi||^{\emph{\textbf{A}}}_{\textbf{M}}=1$, by Axiom \emph{$\forall 1$} of $L\forall^m$, for every ground term $t$, $||\phi(t/x)||^{\emph{\textbf{A}}}_{\textbf{M}}=1$.  By Lemma \ref{Horn substitucio}, $\phi(t/x)$ is also a Horn clause, and since $rk(\phi(t/x))=n$, we can apply the inductive hypothesis and hence for every ground term $t$, $||\phi(t/x)|^{\emph{\textbf{B}}}_{\textbf{N}^{\text{H}}}=1$. Finally, since $\langle\emph{\textbf{B}},\mathbf{N}^{\text{H}}\rangle$ is a Herbrand structure, we have that for every element $t$ of its domain $||\phi(t/x)||^{\emph{\textbf{B}}}_{\textbf{N}^{\text{H}}}=1$, and consequently $||(\forall x)\phi(x)||^{\emph{\textbf{B}}}_{\textbf{N}^{\text{H}}}=1$. \end{proof}

\bigskip

Notice that Proposition \ref{proposition} does not assert that given a structure $\langle\emph{\textbf{A}},\textbf{M}\rangle$,  $\langle\emph{\textbf{A}},\textbf{M}\rangle$ and $\langle\emph{\textbf{B}},\textbf{N}^{\text{H}}\rangle$ satisfy exactly the same equality-free sentences which are Horn clauses. Actually, this is not true. Let $\mathcal{P}$ be a predicate language with three monadic predicate symbols $P_1,P_2,P_3$ and one individual constant $c$.  Suppose that $\emph{\textbf{A}}$ is the \L ukasiewicz algebra $[0,1]_{\text{\L}}$ and let  $\langle\emph{\textbf{A}},\textbf{M}\rangle$ be a structure over $\mathcal{P}$ such that $||P_1(c)||^{\emph{\textbf{A}}}_{\textbf{M}}=1$, $||P_2(c)||^{\emph{\textbf{A}}}_{\textbf{M}}=0.9$ and $||P_3(c)||^{\emph{\textbf{A}}}_{\textbf{M}}=0.5$. Let $\varphi$ be $P_1(c)\& P_2(c)\rightarrow P_3(c)$, $\varphi$ is an equality-free sentence which is a Horn clause with $||P_1(c)\& P_2(c)\rightarrow P_3(c)||^{\emph{\textbf{A}}}_{\textbf{M}}=0.6$, but if we consider its associated H-structure, $\langle\emph{\textbf{B}},\textbf{N}^{\text{H}}\rangle$, we have H$=\{P_1(c)\}$ and thus $||P_1(c)\& P_2(c)\rightarrow P_3(c)||^{\emph{\textbf{B}}}_{\textbf{N}^{\text{H}}}=1$.

\begin{corollary} \label{corollary H-model} An equality-free sentence which is a (w-)Horn clause has a model if and only if it has an \emph{H}-model. 
\end{corollary}

We can conclude here, in the same sense as in Corollary \ref{classic}, that every consistent equality-free sentence which is a (w-)Horn clause has a classical Herbrand model.

\section{Discussion, Conclusions and Future work}  
\label{Conclusions}

The present paper is a first step towards a systematic study of universal Horn theories over predicate fuzzy logics from a model-theoretic perspective. We have proved the existence of free models in universal Horn classes of structures. In the future we will pay special attention to the study of possible characterizations of universal Horn theories in terms of the existence of these free models and its relevance for fuzzy logic programming. 
    
Future work will be devoted also to the analysis of the logical properties of the different definitions of Horn clauses introduced so far in the literature of fuzzy logics, for instance see \cite{BeVic06, BeVic06b, Ma99}. It is important to underline here some differences between our work and some important related references. Our paper differs from the approaches of B\v{e}lohl\'avek and Vychodil and also the one of Gerla, due to mainly three reasons: it is not restricted to fuzzy equalities, it does not adopt the Pavelka-style definition of the Horn clauses and it does not assume the completeness of the algebra. Our choice is taken because it gives more generality to the results we wanted to obtain, even if in this first work our Horn clauses are defined very basically.

We take as a future task to explore how a Pavelka-style definition of Horn clauses in the framework developed by H\'ajek \cite{Ha98} could change or even improve the results we have obtained on free models. We will follow the broad approach taken in \cite[Ch.8]{CiHaNo11} about fuzzy logics with enriched languages. Finally we will study also quasivarieties over fuzzy logic, and closure properties of fuzzy universal Horn classes by using recent results on direct and reduced products over fuzzy logic like \cite{De12}. Our next objective is to solve the open problem  of characterizing theories of Horn clauses in predicate fuzzy logics, formulated by Cintula and H\'ajek in \cite{CiHa10}.

    \section*{Acknowledgments} We would like to thank the referees for their useful comments. This project has received funding from the European Union's Horizon 2020 research and innovation programme under the Marie Sklodowska-Curie grant agreement No 689176 (SYSMICS project). Pilar Dellunde is also partially supported by the project RASO TIN2015-71799-C2-1-P (MINECO/FEDER) and the grant 2014SGR-118 from the Generalitat de Catalunya.

\end{document}